\documentclass[12pt,a4paper]{article}
\usepackage[utf8x]{inputenc}
\usepackage{amsmath}
\usepackage{amsfonts}
\usepackage{amssymb}
\usepackage{amsthm}
\usepackage{mathrsfs}
\usepackage{fancyhdr}
\usepackage{graphicx}
\usepackage[usenames,x11names]{xcolor}
\usepackage{arabtex}
\usepackage{framed}
\usepackage[top=2cm,bottom=2cm,margin=2cm]{geometry}
\usepackage{cancel}
\usepackage{array}   %%%% Pour les fonctionalités assez évoluées d'un tableau

\usepackage[dvipdfm,pdfstartview=FitH,colorlinks=true,linkcolor=Red4,urlcolor=Red4,%
filecolor=green,citecolor=red]{hyperref}

\title{Nontrivial upper bounds for the least common multiple of an arithmetic progression}
\author{\sc Sid Ali BOUSLA \\
Laboratoire de Mathématiques appliquées \\
Faculté des Sciences Exactes \\
Université de Bejaia, 06000 Bejaia, Algeria \\[1mm]
\href{mailto:bouslasidali@gmail.com}{bouslasidali@gmail.com} \\[1mm]
}
\date{}

\def\lcm{\mathrm{lcm}}

\def\EMdash{\leavevmode\hbox to 10.6mm{\vrule height .63ex depth -.59ex
    width 10mm\hfill}}

\theoremstyle{plain}
\numberwithin{equation}{section}
\newtheorem{thm}{Theorem}[section]

\newtheorem{lemma}[thm]{Lemma}

\newtheorem{coll}[thm]{Corollary}

\begin{document}
\maketitle
\begin{abstract}
In this paper, we establish some nontrivial and effective upper bounds for the least common multiple of consecutive terms of a finite arithmetic progression. Precisely, we prove that for any two coprime positive integers $a$ and $b$, with $b\geq 2$, we have
\[\lcm\left(a,a+b,\dots,a+nb\right) \leq \left(c_1\cdot b\log b\right)^{n+\left\lfloor \frac{a}{b}\right\rfloor}~~~~(\forall n\geq b+1),\]
where $c_1=41.30142$. If in addition $b$ is a prime number and $a<b$, then we prove that for any $n\geq b+1$, we have $\lcm\left(a,a+b,\dots,a+nb\right) \leq \left(c_2\cdot b^{\frac{b}{b-1}}\right)^n$, where $c_2=12.30641$. Finally, we apply those inequalities to estimate the arithmetic function $M$ defined by $M(n):=\frac{1}{\varphi(n)}\sum_{\substack{1\leq\ell\leq n \\ \ell \wedge n=1}}\frac{1}{\ell}$ ($\forall n \geq 1$), as well as some values of the generalized Chebyshev function $\theta(x;k,\ell)$.
\end{abstract}

\section{Introduction and Notation}
Throughout this note, we let $\mathbb{N}^*$ denote the set $\mathbb{N} \setminus \{0\}$ of positive integers. The letter $p$ always denotes a prime number and the sequence of all prime numbers is denoted by $\left(p_n\right)_{n\in\mathbb{N}^*}$. We let $\vartheta_p$ denote the usual $p$-adic valuation. We denote by $\lfloor .\rfloor$ the integer-part function. For $x>0$ and $k,\ell\in\mathbb{N}^*$, we set $\pi(x):=\sum_{p\leq x}1$ and $\theta(x;k,\ell):=\sum_{p\leq x}'\log p$, where the sum $\sum'$ is over all the primes in the arithmetic progression of first term $\ell$ and common difference $k$. We also denote by $k\wedge\ell$ the greatest common divisor of $k$ and $\ell$. For $n\in\mathbb{N^*}$, we define $\omega(n):=\sum_{p\mid n}1$ and $M(n):=\frac{1}{\varphi(n)}\sum_{\substack{1\leq\ell\leq n \\ \ell\wedge n=1}}\frac{1}{\ell}$, where $\varphi$ denotes the Euler totient function. For $a,b\in\mathbb{N^*}$, we define $L_{a,b,n}:=\lcm\left(a,a+b,\dots,a+nb\right)$. In order to simplify some statements, we set $c_1:=41.30142$, $c_2:=12.30641$, $c_3:=1.25507$, $c_4:=3.35609$, $c_5:=1.38402$, $c_6:=1.57681$ and $c_7:=2.1284$.

Chebyshev \cite{Chebyshev} was the first to obtain in 1850 an effective bound for the function $\pi$. In particular, he proved that the prime number theorem is equivalent to the statement \linebreak$\log\lcm(1,2,\dots,n)\sim_{+\infty} n$. Recently, several authors are interested to estimate the least common multiple of consecutive terms of arithmetic progressions. Hanson \cite{Hanson} has shown, since 1972, that $\lcm\left(1,2,\dots,n\right)\leq 3^n$ ($\forall n\in\mathbb{N^*}$). In 1982, Nair \cite{Nair} established a simple proof that $\lcm\left(1,2,\dots,n\right)\geq 2^n$ ($\forall n\geq 7$). In the continuation, Farhi \cite{Farhi1,Farhi2} proved that for any coprime positive integers $a$ and $b$, we have
\begin{equation}\label{f1}
L_{a,b,n}\geq a(b+1)^{n-1}~~~~(\forall n\in\mathbb{N}).
\end{equation}
Furthermore, many authors obtained improvements for the last estimate when $n$ is sufficiently large in terms of $a$ and $b$ (see e.g., \cite{Hong1,Hong2,Kane}). In another direction, various asymptotic estimates for the least common multiple of integer sequences have been obtained by several authors; for example, Bateman \cite{Bateman} proved that for any $a,b\in\mathbb{Z}$, with $b>0$, $a+b>0$ and $a\wedge b=1$, we have 
\begin{equation}\label{bat}
\log L_{a,b,n} \sim_{+\infty} bM(b)n.
\end{equation}
For other type of sequences such as polynomial sequences or Lucas sequences, there are also several works that devoted to studying their least common multiple (see e.g., \cite{Bousla,Cilleruelo,Hong3,Hong4,Matiyasevich,Oon}).

In this paper, we establish some nontrivial upper bounds for the least common multiple of consecutive terms of an arithmetic progression. As a consequence, we derive on the one hand effective bounds and an equivalent for $M(n)$ (when $n$ tends to infinity), and on the other hand a nontrivial upper bound for $\theta\left(x;k,\ell\right)$ when $x$, $k$ and $\ell$ satisfy some conditions. It must be noted that although estimates of $\theta\left(x;k,\ell\right)$ exist in the mathematical literature (see e.g., \cite{theta}), the methods of obtaining them are (unlike our own) not elementary in the sense that they do use the theory of analytic functions. Our main results are given in the following:
\begin{thm}\label{t1}
Let $a$ and $b$ be two coprime positive integers, with $b\geq 2$. Then, for any integer $n\geq b+1$, we have
\[L_{a,b,n}\leq \left(c_1\cdot b\log b\right)^{n+\left\lfloor \frac{a}{b}\right\rfloor}.\]
\end{thm}

\begin{thm}\label{t2}
Let $a$ be a positive integer and $b$ be a prime number greater than $a$. Then, for any integer $n\geq b+1$, we have
\[L_{a,b,n}\leq \left(c_2\cdot b^{\frac{b}{b-1}}\right)^n.\] 
\end{thm}

\begin{coll}\label{c1}
For any integer $r\geq 2$, we have
\[\log (r+1)\leq rM(r)\leq \log r+\log\log r +\log c_1.\]
\end{coll}

\noindent The following corollary is immediate.
 
\begin{coll}\label{c2}
We have
\[M(r)\sim_{+\infty}\frac{\log r}{r}.\]
\end{coll}

\begin{coll}\label{c3}
Let $\ell$ be a positive integer and $k$ be a prime number greater than $\ell$. Then, for any real number $x\geq k(k+1)$, we have
\[\theta(x;k,\ell)\leq  x\left(\frac{2c_3}{k}+\frac{\log k}{k-1}\right).\]
\end{coll}

\section{Preliminaries}

\subsection{Previously known results}  

\begin{thm}[Rosser et al. \cite{Rosser}]\label{1}
\noindent
\begin{enumerate}
\item For any integer $n\geq 2$, we have $\sum_{p\leq n}\frac{\log p}{p}\leq \log n$.
\item For any integer $n\geq 6$, we have $p_n\leq n\left(\log n +\log \log n\right)$.
\item The series $\sum_{p}\frac{\log p}{p(p-1)}$ converges to the number $0,7553666111\dots<\log c_7$.\label{3}
\end{enumerate}

\end{thm}
  
\begin{thm}[Hanson \cite{Hanson}]\label{2}
For any real number $x>1$, we have
\[\pi(x)\leq c_3 \frac{x}{\log x}.\]
\end{thm}

\begin{thm}[Robin \cite{Robin}]\label{4}
For any integer $n\geq 3$, we have
\[\omega(n)\leq c_5 \frac{\log n}{\log\log n}.\]
\end{thm}

\subsection{Lemmas}
 
\begin{lemma}\label{6}
Let $a$ and $b$ be two coprime positive integers such that $a<b$. Then, for any integer $n\geq b+1$, we have
\[\prod_{p\leq n}p^{\vartheta_p\left(L_{a,b,n}\right)}\leq \frac{{c_2}^n}{(a+nb)^{\omega(b)}}.\]
\end{lemma}
\begin{proof}
Let $n\geq b+1$ be an integer. First, we remark that for any prime number $p$ dividing $b$, we have $\vartheta_p\left(L_{a,b,n}\right)=0$. Indeed, if $p\mid b$ then $p\nmid a$ (since $a\wedge b=1$) and therefore $p$ does not divide any term of the arithmetic sequence $\left(a+kb\right)_{k\in\mathbb{N}}$; thus $p\nmid L_{a,b,n}$. Next, for any prime number $p$, the number $p^{\vartheta_p\left(L_{a,b,n}\right)}$ is the highest power of $p$ that divides at least one of the numbers $a,a+b,\dots,a+nb$; so, we have $p^{\vartheta_p\left(L_{a,b,n}\right)}\leq a+nb$. Consequently, we have
\[\prod_{p\leq n}p^{\vartheta_p\left(L_{a,b,n}\right)}=\prod_{\begin{subarray}{c} p\leq n \\ p\nmid b \end{subarray}}p^{\vartheta_p\left(L_{a,b,n}\right)}\leq \prod_{\begin{subarray}{c} p\leq n \\ p\nmid b \end{subarray}}(a+nb)=(a+nb)^{\pi(n)-\omega(b)}.\]
Then, since $a+nb\leq n^2$ (because $n\geq b+1>a$), it follows that:
\[\prod_{p\leq n}p^{\vartheta_p\left(L_{a,b,n}\right)}\leq \frac{n^{2\pi(n)}}{(a+nb)^{\omega(b)}}=\frac{e^{2\pi(n)\log n}}{(a+nb)^{\omega(b)}}.\] 
The required estimate then follows from Theorem \ref{2}.
\end{proof}

\begin{lemma}\label{5}
Let $a$ and $b$ be two coprime positive integers. Then, for any natural number $n$, we have
\begin{equation}\label{A}
\prod_{p>n}p^{\vartheta_{p}\left(L_{a,b,n}\right)}~~\text{divides}~~\frac{a\left(a+b\right)\cdots \left(a+nb\right)\cdot \prod_{\begin{subarray}{c} p\leq n \\ p\mid b\end{subarray}}p^{\vartheta_{p}\left(n!\right)}}{n!\cdot\prod_{\begin{subarray}{c} p\leq n \\ p\nmid b\end{subarray}}p^{\vartheta_{p}\left(n+1\right)}}.
\end{equation}
\end{lemma}
\begin{proof}
For $n\in\{0,1\}$, the relation \eqref{A} is trivial. Suppose for the sequel that $n\geq 2$. Let $A_n$ and $B_n$ respectively denote the left-hand side and the right-hand side of \eqref{A}. We will show that $\vartheta_{q}\left(A_n\right)\leq \vartheta_{q}\left(B_n\right)$ for any prime number $q$, which concludes that $A_n$ divides $B_n$. Let $q$ be an arbitrary prime number. In the case where $q$ divides $b$, we have $q\nmid a$ (since $a\wedge b=1$) and thus $q$ does not divide any term of the arithmetic sequence $\left(a+kb\right)_{k\in\mathbb{N}}$, implying that $q$ does not divide $L_{a,b,n}$, and we have therefore
\[\vartheta_{q}\left(B_n\right)=\vartheta_q\left(\frac{\prod_{\begin{subarray}{c} p\leq n \\ p\mid b\end{subarray}}p^{\vartheta_{p}\left(n!\right)}}{n!}\right)=\vartheta_{q}\left(n!\right)-\vartheta_{q}\left(n!\right)=0=\vartheta_{q}\left(A_n\right).\]
It thus remains to show the inequality $\vartheta_{q}\left(A_n\right)\leq \vartheta_{q}\left(B_n\right)$ in the case where $q$ does not divide $b$. Suppose for the sequel that $q\nmid b$ and define $S_{a,b,n}:=\left\lbrace a,a+b,\dots,a+nb\right\rbrace$. We distinguish the following two cases: \\
$\bullet$ \underline{1\textsuperscript{st} case:} (if $q\leq n$). In this case, we have obviously $\vartheta_q\left(A_n\right)=0$. So, we must to show that $\vartheta_q\left(B_n\right)\geq 0$. For any positive integer $\ell$, we let $x_{\ell}$ denote the only solution of the congruence $a+bx\equiv 0 \pmod {q^{\ell}}$ in the set $\{0,1,\dots,q^{\ell}-1\}$. The number of elements of the set $S_{a,b,n}$ which are multiples of $q^{\ell}$ is then equal to the number of integers $x$ such as $0\leq x\leq n$ and $x\equiv x_{\ell} \pmod {q^{\ell}}$; which is clearly equal to $\left\lfloor \frac{n-x_{\ell}}{q^{\ell}}\right\rfloor +1$. Consequently, we have
\begin{align*}
\vartheta_{q}\left(a\left(a+b\right)\cdots \left(a+nb\right)\right)&=\sum_{\ell\geq 1}\left(\left\lfloor \frac{n-x_{\ell}}{q^{\ell}}\right\rfloor +1\right)=\sum_{\ell\geq 1}\left(\left\lfloor \frac{n-x_{\ell}+q^{\ell}}{q^{\ell}}\right\rfloor\right)\\&\geq \sum_{\ell\geq 1}\left\lfloor \frac{n+1}{q^{\ell}}\right\rfloor=\vartheta_{q}\left(\left(n+1\right)!\right).
\end{align*}
Hence $\vartheta_{q}\left(B_n\right)=\vartheta_{q}\left(a\left(a+b\right)\cdots \left(a+nb\right)\right)-\vartheta_{q}((n+1)!)\geq 0$, as required.\\[1mm]
$\bullet$ \underline{2\textsuperscript{nd} case:} (if $q>n$). In this case, since the congruence $a+bx\equiv 0\pmod q$ has exactly one solution in the set $\{0,1,\dots,q-1\}$ (because $b\wedge q=1$) then it has at most one solution in the set $\{0,1,\dots,n\}$. This means that $q$ divides at most one element of the set $S_{a,b,n}$. Consequently, we have $\vartheta_{q}\left(A_n\right)=\vartheta_q\left(L_{a,b,n}\right)=\vartheta_{q}\left(a\left(a+b\right)\cdots \left(a+nb\right)\right)=\vartheta_{q}\left(B_n\right)$. This confirms the required result and completes the proof of the lemma. 
\end{proof}

\begin{lemma}\label{7}
For any integer $b\geq 3$, we have
\[\prod_{p\mid b}{p^{1/p}}\leq c_6 \log b.\]
\end{lemma}
\begin{proof}
If $\omega(b)=1$ then there exists a prime number $q_1$ and a positive integer $m$ such that $b={q_1}^m$. Then we have $\prod_{p\mid b}p^{1/p}={q_1}^{1/{q_1}}\leq 3^{1/3}\leq c_6\log 3\leq c_6\log b$ (since the function $n\mapsto n^{1/n}$ reaches its maximum on the set $\mathbb{N^*}$ at $n=3$). Suppose for the sequel that $\omega(b)\geq 2$ and let us firstly show that:
\begin{equation}\label{lem}
\sum_{p\mid b}\frac{\log p}{p}\leq \sum_{p\leq p_{\omega(b)}}\frac{\log p}{p}.
\end{equation}
In the case where $b$ is even, Inequality \eqref{lem} immediately follows from the decrease of the function $x\mapsto\frac{\log x}{x}$ on the interval $[3,+\infty)$. If now $b$ is odd, then, denoting by $q$ the largest prime factor of $b$, we have (since $\omega(b)\geq 2$ by hypothesis) $q\geq 5$. It follows again from the decrease of the function $x\mapsto\frac{\log x}{x}$ on the interval $[3,+\infty)$ that:
\[\sum_{p\mid b}\frac{\log p}{p}=\sum_{\begin{subarray}{c} p\mid b\\ p\neq q\end{subarray}}\frac{\log p}{p}+\frac{\log q}{q}\leq \sum_{3\leq p\leq p_{\omega(b)}}\frac{\log p}{p}+\frac{\log 5}{5} \leq \sum_{p\leq p_{\omega(b)}}\frac{\log p}{p};\]
confirming \eqref{lem} also in the case where $b$ is odd.\\ Now, by taking the exponential of both sides of \eqref{lem}, we get
\begin{equation}\label{cor}
\prod_{p\mid b}p^{1/p}\leq \prod_{p\leq p_{\omega(b)}}p^{1/p}.
\end{equation}
For $\omega(b)\in\{2,3,4,5\}$, we check by hand that $\prod_{p\leq p_{\omega(b)}}p^{1/p}\leq c_6\log \left(\prod_{p\leq p_{\omega(b)}}p\right)\leq c_6 \log b$, which concludes (according to \eqref{cor}) to the required estimate of the lemma. If, on the contrary, $\omega(b)\geq 6$, then we have (from Theorem \ref{1}):
\[\prod_{p\leq p_{\omega(b)}}p^{1/p}\leq p_{\omega(b)}\leq \omega(b)\left(\log \omega(b) +\log \log \omega(b)\right).\]
Combining this with Theorem \ref{4} and the inequality $\omega(b)\leq \log b$ (which itself derives from Theorem \ref{4} and the fact that $\log b\geq c_5$, since $b\geq 2\cdot 3\cdot 5\cdot 7\cdot 11\cdot 13\geq e^{e^{c_5}}$), we get
\begin{align*}
\prod_{p\leq p_{\omega(b)}}p^{1/p}&\leq c_5 \frac{\log b}{\log \log b}\left(\log c_5+\log \log b -\log\log\log b +\log\log\log b\right)\\&= \left(c_5+\frac{c_5\log c_5}{\log\log b}\right)\log b\\&\leq \left(c_5+\frac{c_5\log c_5}{\log\log\left(2\cdot 3\cdot 5\cdot 7\cdot 11\cdot 13\right)}\right)\log b\\&\leq c_6 \log b.
\end{align*}
Which again concludes (according to \eqref{cor}) to the required estimate of the lemma. This completes the proof.
\end{proof}

\begin{lemma}\label{8}
Let $b\geq 2$ and $n\geq 2$ be two integers. Then, we have
\[\prod_{p\mid b}p^{\vartheta_{p}\left(n!\right)}\leq \left(c_4\log b\right)^{n}.\]
\end{lemma}  
\begin{proof}
For $b=2$, the estimate of the lemma immediately follows from Legendre's formula. Indeed, we have
\[\prod_{p\mid 2}p^{\vartheta_{p}\left(n!\right)}=2^{\vartheta_{2}\left(n!\right)}=2^{\lfloor \frac{n}{2}\rfloor+\lfloor \frac{n}{4}\rfloor+\lfloor \frac{n}{8}\rfloor+\dots}\leq 2^{\frac{n}{2}+\frac{n}{4}+\frac{n}{8}+\dots}=2^{n}\leq \left(c_4\log 2\right)^{n}.\]
Now, suppose that $b\geq 3$. By using successively Legendre's formula, Lemma \ref{7} and the point \ref{3} of Theorem \ref{1}, we have
\begin{align*}
\prod_{p\mid b}p^{\vartheta_{p}\left(n!\right)}&=\prod_{p\mid b}p^{\left\lfloor\frac{n}{p}\right\rfloor + \left\lfloor\frac{n}{p^2}\right\rfloor +\dots}\leq \prod_{p\mid b}p^{\frac{n}{p} + \frac{n}{p^2} +\dots}=\left(\prod_{p\mid b}p^{\frac{1}{p}}\right)^n \left(\prod_{p\mid b}p^{\frac{1}{p(p-1)}}\right)^n\\&\leq \left(c_6 c_7\log b\right)^{n}\leq \left(c_4\log b\right)^{n},
\end{align*}
as required. This completes the proof of the lemma.
\end{proof}

\section{Proofs of our main results}

\begin{proof}[Proof of Theorem \ref{t1}]
Let $n\geq b+1$ be an integer and let us first assume that $a<b$. By using successively Lemmas \ref{6}, \ref{5} and \ref{8}, we have
\begin{align*}
L_{a,b,n}&=\left(\prod_{p\leq n}p^{\vartheta_p\left(L_{a,b,n}\right)}\right)\left(\prod_{p>n}p^{\vartheta_p\left(L_{a,b,n}\right)}\right)\\&\leq \frac{{c_2}^n}{(a+bn)^{\omega(b)}}\cdot\frac{a\left(a+b\right)\cdots \left(a+nb\right)}{n!}\cdot \prod_{\begin{subarray}{c} p\leq n \\ p\mid b\end{subarray}}p^{\vartheta_{p}\left(n!\right)}\\ &\leq {c_2}^n\frac{(a+nb)}{(a+nb)^{\omega(b)}}\cdot \frac{b(2b)(3b)\cdots \left(nb\right)}{n!}\cdot\prod_{p\mid b}p^{\vartheta_{p}\left(n!\right)}\\&\leq {c_2}^n b^n\left(c_4\log b\right)^n\leq \left(c_1\cdot b\log b\right)^n,
\end{align*}
as required. Next, if $a>b$, then by setting $q:=\left\lfloor a/b\right\rfloor$ and $a':=a-qb<b$, we have obviously $L_{a,b,n}$ divides $L_{a',b,q+n}$; which leads to the required result by applying the first case to the triplet $(a',b,q+n)$ instead of $(a,b,n)$. This completes the proof of the theorem.
\end{proof}

\begin{proof}[Proof of Theorem \ref{t2}]
It suffices to repeat the previous proof of Theorem \ref{t1} (the first case precisely) and to use the following estimate
\[\prod_{\begin{subarray}{c} p\leq n \\ p\mid b\end{subarray}}p^{\vartheta_{p}\left(n!\right)}=b^{\vartheta_{b}\left(n!\right)}=b^{\left\lfloor\frac{n}{b}\right\rfloor + \left\lfloor\frac{n}{b^2}\right\rfloor +\dots}\leq b^{\frac{n}{b-1}}\]
instead of that of Lemma \ref{8}.
\end{proof}

\begin{proof}[Proof of Corollary \ref{c1}]
Given $r\geq 2$ an integer and $n\geq r+1$ an integer, we have by using Estimate \eqref{f1} and then Theorem \ref{t1}:
\[\frac{n-1}{n}\log (r+1)\leq \frac{\log L_{1,r,n}}{n}\leq \log r+\log\log r +\log c_1.\]
The required result follows by taking the limits of both sides of the last double inequality, as $n$ tends to infinity, and use Estimate \eqref{bat}.
\end{proof}

\begin{proof}[Proof of Corollary \ref{c3}]
Let $x\geq k(k+1)$ and $m:=\left\lfloor \frac{x}{k}\right\rfloor$. We have clearly:
\[\theta(x;k,\ell)\leq \log\lcm\left(\ell,\ell+k,\dots,\ell+mk\right)=\log L_{\ell,k,m}.\]
On the other hand, since $m\geq k+1$ (because $x\geq k(k+1)$), we have (according to Theorem \ref{t2}):
\[\log L_{\ell,k,m}\leq m\left(\log c_2+\frac{k\log k}{k-1}\right)\leq x\left(\frac{2c_3}{k}+\frac{\log k}{k-1}\right),\]
which concludes to the required result.
\end{proof}

\end{document}